\documentclass{interact}

\usepackage{epstopdf}
\usepackage{subfigure}

\theoremstyle{plain}
\newtheorem{theorem}{Theorem}[section]

\newtheorem{corollary}[theorem]{Corollary}

\newtheorem{proposition}[theorem]{Proposition}

\theoremstyle{remark}
\newtheorem{remark}{Remark}

\theoremstyle{definition}

\newcommand{\tends}[1]{\xrightarrow[#1]{}}
\newcommand{\tendsd}{\xrightarrow{\ d\ }}
\newcommand{\tendsp}{\xrightarrow{\ P\ }}

\DeclareMathOperator{\Var}{D}

\DeclareMathOperator{\Prob}{P} 
\DeclareMathOperator{\Mean}{E}

\newcommand{\abs}[1]{\lvert #1\rvert}

\newcommand{\Nd}{\mathbb{N}}
\newcommand{\Rd}{\mathbb{R}}
\newcommand{\Cd}{\mathbb{C}}

\newcommand{\calF}{\mathcal{F}}

\newcommand{\e}{\mathrm{e}}

\newcommand{\imply}{\Rightarrow}

\begin{document}


\title{On asymptotic normality of certain linear rank statistics}

\author{
\name{V. Skorniakov\textsuperscript{a}$^{\ast}$\thanks{$^\ast$Corresponding author. Email: viktor.skorniakov@mif.vu.lt}}
\affil{\textsuperscript{a}Vilnius University, Faculty of Mathematics and Informatics, Naugarduko 24,
LT-03225, Vilnius, Lithuania}
}

\maketitle

\begin{abstract}
We consider asymptotic normality of linear rank statistics under various randomization rules met in clinical trials and designed for patients' allocation into treatment and placebo arms. Exposition relies on some general limit theorem due to \cite{McLeish74} which appears to be well suited for the problem considered and may be employed for other similar rules undiscussed in the paper. Examples of applications include well known results as well as several new ones.
\end{abstract}


\begin{classcode}62G10; 62G20; 62P10\end{classcode}

\section{Introduction}

In order to  adequately measure an effect of treatment it is common a practice in clinical trial to randomize patients into those receiving tested treatment and those receiving placebo or standard therapy. To achieve the goals of the study different randomization rules may be applied. In what follows we consider the ones randomizing into two groups sequentially and intended to produce treatment and placebo groups of approximately equal sizes.  Each such rule may be described as follows. Let $n$ denotes a total number of patients to be randomized\footnote{for the sake of convenience we assume that $n$ is even} whereas $T_{n,j}$ takes value 1 in case $j$-th patient was assigned to receive investigated therapy and value $-1$ provided it was on the contrary. Then the rule is defined by conditional probabilities \[
\Prob(T_{n,j}=t_j \mid T_{n,j-1}=t_{j-1},\dots,T_{n,1}=t_1),
\]
according to which actual randomization takes place in practice. Several popular rules considered in the sequel are given in table \ref{tbl:rules}.

\begin{table}
\tbl{Several popular randomization rules.}
{\begin{tabular}[l]{@{}lc}\toprule
  \textbf{Common name of the rule} &  $\mathbf{\Prob(T_{n,j} = 1 \mid T_{n,j-1} = t_{j-1}, \dots,T_{n,1} = t_1)),j >1} $$^{\rm a,b}$
          \\
\colrule
    Complete randomization & $1/2$  \\ \hline
    Random allocation & $\max\left(0, \frac{\frac{n}{2}-S^{(1)}_{n,j-1}}{n-(j-1)}\right)$   \\ \hline
      & $1/2$, if $\max(S^{(-1)}_{n,j-1},S^{(1)}_{n,j-1})<\frac{n}{2}$;\\
     Truncated binomial design & $1$, if $S^{(-1)}_{n,j-1} = \frac{n}{2}$;\\ 
      & $0$, if $S^{(1)}_{n,j-1} = \frac{n}{2}$.\\ \hline
     & $1/2$, if $S^{(1)}_{n,j-1} = S^{(-1)}_{n,j-1}$;\\
     Wei's urn design$^{\rm c}$ $U(\alpha,\beta)$ &  $\frac{\alpha + \beta S^{(-1)}_{n,j-1}}{2 \alpha + \beta(j-1)}$   \\ 
\botrule
\end{tabular}}
\tabnote{$^{\rm a}$ in all cases $\Prob(T_{n,1}=\pm)=1/2$, $S^{(1)}_{n,j}=\sum\limits_{1\leq k\leq j:t_{k}=1}t_k=$'size of the treatment group after occurrence of $j$ assignments', $S^{(-1)}_{n,j}=j-S^{(1)}_{n,j}$}
\tabnote{$^{\rm c}$ $\alpha,\beta\in \Nd_0$ are some fixed known constants defining the rule}
\label{tbl:rules}
\end{table}

Let $Y_j$ denotes an outcome of $j$-th patient measured on continuous scale. One can apply different sample models upon which an inference is built and conclusion about the presence or absence of the treatment effect is made. The linear rank statistics (see \cite{Lachin02}, \cite{Lachin16}) is one of possible choices. To construct statistics of this type one should proceed as follows:
\begin{itemize}
\item given realization $y_1,\dots,y_n$ of $Y_1,\dots,Y_n$ associate with each $y_j$ the score\footnote{one of possible and frequent choices is to take $a_{n,j}$ equal to a simple rank obtained after ranking $y_1,\dots,y_n$; other popular choices of scores are given in table \ref{tbl:scores}} $a_{n,j}$;
\item put 
\begin{equation}\label{e:Ln}
	L_n=\sum_{j=1}^n(a_{n,j}-\bar{a}_n)T_{n,j},\quad \bar{a}_n=\frac{1}{n}\sum_{j=1}^na_{n,j};
\end{equation}
\item consider $a_n=(a_{n,1},\dots,a_{n,n})^T$ as fixed and $T_n=(T_{n,1},\dots,T_{n,n})^T$ as random.
\end{itemize}
Then $L_n$ is a linear rank statistics.

For a fixed randomization rule the distribution of $L_n$ is easy to tabulate provided sample size $n$ is relatively small, however, for a big one asymptotic result may be a good alternative. In the present paper we discuss  conditions under which appropriately centered and scaled $L_n$ tends to standard normal variate for the rules listed in table \ref{tbl:rules}. The whole exposition grounds on some general theorem due to \cite{McLeish74}. It is restated in section \ref{s:GLimTh} with relevant comments. Section \ref{s:Ex} is devoted to the above mentioned examples illustrating an adoption of the result given in \cite{McLeish74} for the case of linear rank statistics \eqref{e:Ln}. We are inclined to think that one can proceed in a similar way when considering other rules similar to those listed in table \ref{tbl:rules}. Finally section \ref{s:proofs} contains proofs of several propositions stated in sections \ref{s:GLimTh} and \ref{s:Ex}.

\begin{table}
\tbl{Several frequent choices of scores.}
{\begin{tabular}[l]{@{}lc}\toprule
  \textbf{Common name of the scores} & Formula for$^{\rm a}$ $a_{n,j}$
          \\ \colrule
	Median scores & $a_{n,j}=\mathbf{1}_{\left\{\left(\frac{n+1}{2};\infty\right)\right\}}(r_{n,j}),j=1,\dots,n$\\ \hline
    Wilcoxon scores & $a_{n,j}=r_{n,j}$\\ \hline
    van der Waerden scores & $a_{n,j}=\Phi^{-1}\left(\frac{r_{n,j}}{n+1}\right)$, where $\Phi$  is a c.d.f. of 	$N(0;1)$ r.v.\\ \hline
    Savage scores & $a_{n,j}=\sum_{k=1}^{r_{n,j} }\frac{1}{n-k+1}-1$ \\
\botrule
\end{tabular}}
\tabnote{$^{\rm a}$ $r_{n,j}$ denote the simple ranks of $Y_j$ obtained after ranking the whole realization of the sample $Y_1,\dots,Y_n$}
\label{tbl:scores}
\end{table}

\begin{table}
\tbl{Expressions for conditional expectations.}
{\begin{tabular}[l]{@{}lc}\toprule
  \textbf{Name of the rule} &  $\mathbf{\Mean (T_{n,j} \mid T_{n,j-1}, \dots,T_{n,1}),j >1} $$^{\rm a,b}$
          \\
\colrule
    Complete randomization & $0$  \\ \hline
    Random allocation & $\frac{{n}-{2}S^{(1)}_{n,j-1}}{n-(j-1)}$   \\ \hline      
     Truncated binomial design & $\mathbf{1}_{\frac{n}{2}}(S^{(-1)}_{n,j-1})-\mathbf{1}_{\frac{n}{2}}(S^{(1)}_{n,j-1})$\\ \hline
     Wei's urn design$^{\rm c}$ $U(\alpha,\beta)$ &  $\frac{2\beta( S^{(-1)}_{n,j-1}-(j-1))}{2 \alpha + \beta(j-1)}$   \\ 
\botrule
\end{tabular}}
\tabnote{$^{\rm a}$ in all cases $\Mean T_{n,1}=0$, $S^{(1)}_{n,j}=\sum\limits_{1\leq k\leq j:T_{n,k}=1}T_{n,k}=$'size of treatment group after occurrence of $j$ assignments', $S^{(-1)}_{n,j}=j-S^{(1)}_{n,j}$}
\tabnote{$^{\rm c}$ $\alpha,\beta\in \Nd_0$ are some fixed known constants defining the rule}
\label{tbl:expectations}
\end{table}

\section{Auxiliary results}\label{s:GLimTh}

In his seminal paper of 1974 Don L. McLeish (see \cite{McLeish74}) proved the following theorem.
\begin{theorem}\label{t:McLeish74}
Let $(r_n)\subset \Nd$ be an increasing sequence and let $\{Z_{n,j}\mid j=1,\dots,r_n,n\in \Nd\}$ be a zero mean stochastic array. Put 
\begin{equation*}
\pi_n=\prod\limits_{j=1}^{r_n}(1+\mathrm{i}\lambda Z_{n,j}),\quad \lambda>0,\quad \mathrm{i}=\sqrt[]{-1},\quad S_n=\sum\limits_{j=1}^{r_n} Z_{n,j}.	
\end{equation*}
Assume the following:
\begin{itemize}
\item[(a)] $\forall \lambda >0\ \pi_n$ is uniformly integrable;
\item[(b)] $\forall \lambda >0\ \Mean\pi_n \tends{n\to\infty} 1$;
\item[(c)] $\sum\limits_{j=1}^{r_n} Z_{n,j}^2 \tendsp 1$;
\item[(d)] $\max\limits_{1\leq j\leq r_n} \abs{Z_{n,j}} \tendsp 0$.
\end{itemize}
Then $S_n \tendsd N(0;1)$.
\end{theorem}
\noindent The stated theorem appears to be very well suited to handle the case of linear rank statistics given by \eqref{e:Ln} provided\footnote{it holds true for all the rules listed in table \ref{tbl:rules} since $\Prob(T_{n,j}=\pm 1)=1/2$; for corresponding derivations see \cite{Lachin02} or \cite{Lachin16}} $\Mean T_{n,j}=0,j=1,\dots,n$. To see this put
\begin{gather}
	s_{n,j}=\frac{a_{n,j}-\bar{a}_n}{\sqrt[]{\sum_{j=1}^n\left(a_{n,j}-\bar{a}_n\right)^2}},\nonumber \\
	S_n=\sum_{j=1}^n s_{n,j}T_{n,j}=\left[Z_{n,j}=s_{n,j}T_{n,j}\right]=\sum_{j=1}^{n}Z_{n,j}. \label{e:Sn}
\end{gather}
Then $\forall j \ \Mean Z_{n,j}=0$ and condition (c) holds trivially. (d) reads as
\begin{equation}\label{e:const1}
\max_j\abs{s_{n,j}}=\frac{\abs{a_{n,j}-\bar{a}_n}}{\sqrt[]{\sum_{j=1}^n\left(a_{n,j}-\bar{a}_n\right)^2}}\to 0
\end{equation}
and is a natural restriction in problems of this kind. Therefore we assume that it holds for all examples considered in section \ref{s:Ex}. For justification consider the simplest case of complete randomization: to prove asymptotic result given in example \ref{ss:completeRandAlloc} by making use of Lindeberg CLT one should necessary impose constraint \eqref{e:const1}. Next, note that (a) also holds. Indeed, fix $\lambda>0$. Then by mean value theorem,
\begin{multline*}
	\ln\abs{\pi_n}^2=\sum_{j=1}^n \ln\left(1+\lambda^2 Z_{n,j}^2\right)=
    \sum_{j=1}^n \ln\left(1+\lambda^2 Z_{n,j}^2\right)-\ln 1=\\
    \sum_{j=1}^n \frac{\lambda^2 Z_{n,j}^2}{1+\theta_{n,j}\lambda^2 Z_{n,j}^2}\leq
    \lambda^2\sum_{j=1}^n {Z_{n,j}^2}=\lambda^2,
\end{multline*}
for some $\theta_{n,j}\in (0;1)$. Summing up, under constraint \eqref{e:const1}, (b) is the only condition one needs to check for $S_n$ given by \eqref{e:Sn} to satisfy $S_n \tendsd N(0;1)$.

For certain rules, however, it is more convenient to make use of the following result stemming from theorem \ref{t:McLeish74}.

\begin{theorem}[\cite{Davidson94}, Theorem 24.3]\label{t:McLeishCorr}
Let $\{(X_{n,j},\calF_{n,j})\mid j=1,\dots,r_n \uparrow \infty, r_n\in \Nd,n\ge 1 \}$ be a martingale difference array\footnote{that is, $(\calF_{n,j})$ is non-decreasing sequence of $\sigma$-algebras, $\forall n,j\ X_{n,j}$ is $\calF_{n,j}$ measurable and $\Mean(X_{n,j}\mid \calF_{n,j-1})=0$} with finite unconditional variances $\sigma_{n,j}^2$ such that $\sum_{j=1}^n \sigma^2_{n,j}=1$. If
\begin{itemize}
\item[(a)] $\sum_{j=1}^{r_n}X_{n,j}^2 \tendsp 1$ and
\item[(b)] $\max_{1\leq j \leq r_n}\abs{X_{n,j}}\tendsp 0$,
\end{itemize}
then $\sum_{j=1}^{r_n}X_{n,j} \tendsd N(0;1)$.
\end{theorem}
\noindent Retain the notions introduced and set\footnote{here and further on $\Var(X)$ denotes a variance of $X$} 
\begin{gather}
\calF_{n,j}=\sigma\left(\left\{T_{k,i}\mid i=1,\dots,k; k=1,\dots,n-1\right\} \cup\left\{T_{n,1},\dots,T_{n,j}\right\}\right),\label{e:Fnj} \\
\tilde{Z}_{n,j}=s_{n,j}\left(T_{n,j}-\Mean(T_{n,j} \mid \calF_{n,j-1})\right), \quad X_{n,j}=
\frac{\tilde{Z}_{n,j}}{\sqrt[]{\sum_{j=1}^n \Var\tilde{Z}_{n,j}}}. \nonumber
\end{gather}
Since considered randomization rules are intended to produce groups of approximately equal sizes, it is natural to expect that the rule from this class will pretty often have the property
\begin{multline*}
	0 \xleftarrow[n,j\to\infty]{\ P\ } \Mean(T_{n,j}\mid \calF_{n,{j-1}})=2\Prob(T_{n,j}=1 \mid T_{n,1},\dots,T_{n,j-1})-1 \Longleftrightarrow\\ 
    \Prob(T_{n,j}=1 \mid T_{n,1},\dots,T_{n,j-1}) 
     \xrightarrow[n,j\to\infty]{\ P\ }\frac{1}{2},
\end{multline*}
rigorously read by us as follows:
\begin{equation}\label{e:vanishingConditional}
	\forall \epsilon,\delta>0\ \exists n_{\epsilon,\delta}>0\ \forall n,j \geq n_{\epsilon,\delta}
    \Prob\left( \abs{\Mean(T_{n,j}\mid \calF_{n,j-1})}<\epsilon\right)\geq 1-\delta.
\end{equation}
Assume it holds. Then the following is true\footnote{for the proof see section \ref{s:proofs}}.
\begin{proposition}\label{prop:vanishingConditional}
Under \eqref{e:vanishingConditional} and \eqref{e:const1},
\begin{equation*}
	\sum_{j=1}^n s_{n,j}(T_{n,j}-\Mean(T_{n,j}\mid \calF_{n,j-1})) \tendsd N(0;1).
\end{equation*}
\end{proposition}

\noindent Given proposition is well suited for applications since the centering terms $\Mean(T_{n,j}\mid \calF_{n,j-1})=2\Prob(T_{n,j}=1 \mid T_{n,1},\dots,T_{n,j-1})-1$ involve only conditional probabilities defining the randomization rule and thus do not require any extra calculations or assumptions.

In case when $\Mean(T_{n,j} \mid \calF_{n,j-1})\not\to 0$, one can still apply theorem \ref{t:McLeishCorr} provided computation of the variance in denominator is easy. This time, however, conditions (a) and (b) need verification and in general may impose additional constraints on $\{a_{n,j}\}$ beyond that given by \eqref{e:const1} as it shown by example \ref{p:TBD}.

Finishing this section we summarize the constraints one needs to impose/verify and corresponding implications by making use of the notions introduced in this subsection. This will help when going through the proofs of propositions given in section \ref{s:Ex}. 

\subsection{Constraints}\label{ss:auxConstraints}
\begin{itemize}
\item[(c1)] $\max\limits_{1\leq j\leq n}\abs{s_{n,j}}=\frac{\abs{a_{n,j}-\bar{a}_n}}{\sqrt[]{\sum_{j=1}^n}\left(a_{n,j}-\bar{a}_n\right)^2}\tends{n\to\infty} 0$;
\item[(c2)] $\forall n,j\quad \Mean T_{n,j}=0$;
\item[(c3)] $\forall \lambda>0\quad \Mean \pi_{n}\tends{n\to\infty}1$ with $\pi_n=\prod_{j=1}^{n}(1+\mathrm{i}\lambda Z_{n,j})$ and $Z_{n,j}=s_{n,j}T_{n,j}$;
\item[(c4)] $\Mean(T_{n,j}\mid \calF_{n,j-1}) \xrightarrow[n,j\to\infty]{\ P\ }0$ with $\calF_{n,j}$ defined by \eqref{e:Fnj} and in the sense of \eqref{e:vanishingConditional}.
\end{itemize}

\subsection{Implications}\label{ss:auxImplications}
\begin{itemize}
\item[(i)] (c1), (c2), c(3) $\imply \sum_{j=1}^n Z_{n,j} \tendsd N(0;1)$;
\item[(ii)] (c1), (c4) $\imply \sum_{j=1}^n \tilde{Z}_{n,j} \tendsd N(0;1)$, with $\tilde{Z}_{n,j}=s_{n,j}\left(T_{n,j}-\Mean(T_{n,j} \mid \calF_{n,j-1})\right)$.
\end{itemize}

\subsection{Concluding remarks}\label{ss:auxConcludingRems}
\begin{itemize}
\item[(r1)] through the rest part of the paper we retain all notions introduced in this section including those given in tables; 
\item[(r2)] it was already mentioned that (c2) holds for all rules listed in table \ref{tbl:rules}; therefore for (i) to hold one only needs to verify (c3);
\item[(r3)] dealing with a particular rule we apply combination of constraints which seems most convenient and/or least restrictive for that particular rule;
\item[(r4)] note that all scores given in table \ref{tbl:scores} satisfy (c1) (see table\footnote{to fill the table one has to produce some simple but tedious calculation; we therefore omit this process} \ref{tbl:scoresA}).
\end{itemize}

\begin{table}
\tbl{Asymptotic properties of scoring rules$^{\rm a}$.}
{\begin{tabular}[l]{@{}lc}\toprule
  \textbf{Name of the scoring} & \textbf{Order of} 
  \\
 \textbf{rule} & $\mathbf{\max_{1\leq j\leq n} s^2_{n,j}}$ 
          \\ \colrule
	Median scores & $O\left(\frac{1}{n}\right)$ \\ \hline
    Wilcoxon scores & $O\left(\frac{1}{n}\right)$ \\ \hline
    van der Waerden scores$^{\rm b,c}$ & $O\left(\frac{\ln n}{n}\right)$ \\ \hline
    Savage scores & $O\left(\frac{\ln^2 n}{n}\right)$ \\
\botrule
\end{tabular}}
\tabnote{$^{\rm a}$ $r_{n,j}$ denote the simple ranks of $Y_j$ obtained after ranking the whole realization of the sample $Y_1,\dots,Y_n$}
\tabnote{$^{\rm b}$ $\Phi^{-1}$ denotes an inverse of the c.d.f. of the standard normal variate $N(0;1)$}
\tabnote{$^{\rm c}$ one may require to make use of asymptotic approximations for quantile of $N(0;1)$ provided in any standard reference similar to \cite{patel_handbook_1996}}
\label{tbl:scoresA}
\end{table}

\section{Examples}\label{s:Ex}

In this section we provide three examples devoted to illustrate three approaches of application of the general theorems of section \ref{s:GLimTh}.

\subsection{Complete randomization and random allocation rule}\label{ss:completeRandAlloc}

For the case of complete randomization and random allocation rules the following applies.

\begin{proposition}\label{prop:completeRandAlloc}
Let $T_{n,1},\dots,T_{n,n}$ be a randomization sample corresponding to complete randomization or random allocation rule. Assume (c1). Then 
\begin{equation}\label{e:simplestConv}
	\sum_{j=1}^n Z_{n,j} \tendsd N(0;1).
\end{equation}
\end{proposition}

The result given above is well known\footnote{for an alternative proof different from that of ours see \cite{Lachin02}} and included here only for the sake of demonstration of application of theorem \ref{t:McLeish74}.

\subsection{Wei's Urn design $U(\alpha,\beta)$}\label{p:WeisUrnDesign}

Let $S_{n,j}^{(k)}=\sum_{l=1}^{j}\mathbf{1}_{\{k\}}(T_{n,l}),k=\pm 1,j=1,\dots,n$. From table \ref{tbl:expectations} it follows that
\begin{equation*}
	\Mean(T_{n,j}\mid \calF_{n,j-1})=\frac{\beta( 2S^{(-1)}_{n,j-1}-(j-1))}{2 \alpha + \beta(j-1)}=
    \frac{\frac{S^{(-1)}_{n,j-1}}{j-1}-\frac{1}{2}}{\frac{1}{2}+ \frac{\alpha}{\beta(j-1)}} \tendsp 0,
\end{equation*}
since the law of large numbers applies to\footnote{see \cite{Wei86}} $S_{n,j}^{(k)}$ and $\Mean \mathbf{1}_{\{k\}}(T_{n,j})=\Prob(T_{n,j}=k)=1/2,k=\pm1,j=1,\dots,n$. Hence, under (c1) proposition \ref{prop:vanishingConditional} applies and we immediately obtain the proposition below.

\begin{proposition}\label{prop:WeisDesign}
Let $T_{n,1},\dots,T_{n,n}$ be a randomization sample corresponding to Wei's urn design. Then (c1) implies \eqref{e:simplestConv}.
\end{proposition}

An asymptotic linear rank test involving scores and based on randomization of this kind was investigated in \cite{Wei83} and \cite{Wei86}. The authors also made use of martingale theory. It is instructive to note a gain in the ease of proof provided by our approach as well as computational difficulty of statistic suggested in \cite{Wei83}. Empirical findings reported in \cite{Lachin02}, page 237, suggest that our statistic should perform more-or-less alike as that of \cite{Wei83}. However, we do not provide any simulational results to support this opinion since our purpose here lies only in demonstration of derivations.

\subsection{Truncated binomial design}\label{p:TBD}

This design, seeming pretty simple at first glance, represents an interesting case of restricted randomization rule\footnote{that is, when randomization is finished placebo and control arms contain equal numbers of patients} and deserves special attention. To derive conditions ensuring asymptotic normality we make direct use of theorem \ref{t:McLeishCorr} combined with stopping technique. Our main result is contained in the following proposition.

\begin{proposition}\label{prop:TBDSum}
Let $T_{n,1},\dots,T_{n,n}$ denote the randomization sample corresponding to truncated binomial design and let 
\begin{equation*}
	\tau_n=\tau=n-\min\left\lbrace j\in\{n/2,\dots,n-1\}\Bigg| \max\left(\sum_{k=1}^{j}\mathbf{1}_{\{1\}}(T_{n,k}),\sum_{k=1}^{j}\mathbf{1}_{\{-1\}}(T_{n,k})\right)=\frac{n}{2}\right\rbrace.
\end{equation*}
Assume (c1),
\begin{itemize}
\item[(i)] $\liminf_{n\to\infty}\left(\sum_{j=1}^{n/2} s_{n,j}^2+ \sum_{j=n/2+1}^{n} s_{n,j}^2\Prob(\tau \leq n-j)\right)>0$ and
\item[(ii)] $ \sum_{j=n/2+1}^{n-\tau} s_{n,j}^2\Prob(\tau>n-j)-\sum_{j=n-\tau+1}^{n} s_{n,j}^2 \tendsp 0$.
\end{itemize}
Then $\frac{\sum_{j=1}^{\tau} s_{n,j}T_{n,j}}{\sqrt{\sum_{j=1}^{n/2} s_{n,j}^2+ \sum_{j=n/2+1}^{n} s_{n,j}^2\Prob(\tau \leq n-j)}}\tendsd N(0;1)$.
\end{proposition}

\begin{remark}\label{rem:onTau}
By definition of design it turns out that once one of the groups has achieved its maximal capacity $\frac{n}{2}$, the rest assignments in the tail are all taken equal to that of unfilled group. Note that $\tau \in \{1,\dots,n/2\}$ denotes a r.v. equal to the size of such tail assignment. The distribution of $\tau$ is given by the set of equations (see \cite{Lachin02}, subsection 3.4) $\Prob(\tau = k)=\frac{1}{2^{n-k-1}}\binom{n-k-1}{n/2-1},k=1,\dots,\frac{n}{2}.  \qquad\blacktriangle$
\end{remark}

\noindent Though a distribution of $\tau$ is explicitly known, conditions given above seem unhandy. Therefore below we provide "ready to apply" simplification.

\begin{proposition}\label{prop:TBDSum2}
Consider the setting of proposition \ref{prop:TBDSum}. Then
\begin{itemize}
\item[(s1)] $\max_{n/2\leq j \leq n} s_{n,j}^2=o\left(\frac{1}{\sqrt[]{n}}\right)\implies$ (i);
\item[(s2)] $\max_{n/2\leq j \leq n} s_{n,j}^2=o\left(\frac{1}{\sqrt[]{n\ln n}}\right)\implies$ (ii).
\end{itemize}
\end{proposition}

\noindent A direct application of this combined with information given in table \ref{tbl:scoresA} leads to the following corollary.

\begin{corollary}\label{cor:TBDAllScores}
Proposition \ref{prop:TBDSum} applies to all arrays of scores given in table \ref{tbl:scores}. 
\end{corollary}

The case of truncated binomial design was treated in \cite{Rosenberger03} and \cite{Rosenberger05}. In the latter paper the authors pointed out that obtained statistic exhibited better properties than that of \cite{Rosenberger03}. Inspection of the proofs shows that\footnote{in the original statement the norming denominator expressed in terms of $a_{n,j}-\bar{a}_n$ is a bit different, however, in the body of the proof of the main theorem the authors show its asymptotic equivalence to $\sqrt[]{\sum_{j=1}^n(a_{n,j}-\bar{a}_n)^2}$} in our notation main result given there reads as follows (\cite{Rosenberger05}, theorem 1).
\begin{theorem}\label{t:TBDmainOfRosenberger}
In addition to (c1) assume the following:
\begin{itemize}
\item[(i)] $\max_{n/2 \leq j\leq n}s_{n,j}^2=o\left(\frac{1}{\sqrt[]{n}}\right)$;
\item[(ii)] $\exists \delta_1>0,\delta_2\in(1/2;1)$ such that $\max\limits_{j \geq \frac{n}{2}-\delta_1 n^{\delta_2}}\left(\sum_{k=n/2+j}^ns_{n,k}\right)^2 =o\left(\frac{1}{{n}^{\delta_2-1/2}}\right)$.
\end{itemize}
Then $\sum_{j=1}^n s_{n,j}T_{n,j} \tendsd N(0;1)$.
\end{theorem}
One can see that the above constraints put the main weight of $s_{n,1},\dots, s_{n,n}$ to the fore half $s_{n,1},\dots,s_{n,n/2}$ making the tail half $s_{n,n/2+1},\dots,s_{n,n}$ light enough. As a consequence, the authors show that the theorem does not apply to Savage scores. Corollary \ref{cor:TBDAllScores}, however, does not exclude Savage scores. Taking this into account as well as pretty handy conditions given in proposition \ref{prop:TBDSum2}  we may view results of this subsection as an improvement of both \cite{Rosenberger03} and \cite{Rosenberger05}. It is, however, honest dealing to note that we did not take any effort to show that conditions of theorem \ref{t:TBDmainOfRosenberger} imply the ones stated in proposition \ref{prop:TBDSum}. Hence, formally the question whether it is true remains open.

\section{Proofs}\label{s:proofs}

\begin{proof}[Proof of proposition \ref{prop:vanishingConditional}.]
Assume \eqref{e:vanishingConditional}. Fix $\epsilon>0$ and find corresponding $n_{\epsilon,\epsilon}$. Then for $n,j \geq n_{\epsilon,\epsilon}$,
\begin{multline*}
	\Var(\tilde{Z}_{n,j})=s_{n,j}^2\Mean\left(T_{n,j}-\Mean(T_{n,j}\mid \calF_{n,j-1})\right)^2=\\
    s_{n,j}^2\left(1-2\Mean\left( T_{n,j}\Mean(T_{n,j}\mid \calF_{n,j-1})\right)+\Mean\left(\Mean(T_{n,j}\mid \calF_{n,j-1})\right)^2\right)=\\
    \left[ r_{n,j}= 
    s_{n,j}^2\Mean(T_{n,j}\mid \calF_{n,j-1})\left(\Mean(T_{n,j}\mid \calF_{n,j-1})-2 T_{n,j}\right)\right]=
    s_{n,j}^2+\Mean r_{n,j},
\end{multline*}
with
\begin{equation*}
	\Mean\abs{r_{n,j}}=\Mean\boldmath{1}_{\left\lbrace
    \abs{\Mean(T_{n,j}\mid \calF_{n,j-1})}>\epsilon
    \right\rbrace}\abs{r_{n,j}}+\Mean\boldmath{1}_{\left\lbrace
    \abs{\Mean(T_{n,j}\mid \calF_{n,j-1})}\leq\epsilon
    \right\rbrace}\abs{r_{n,j}}\leq 6\epsilon s_{n,j}^2.
\end{equation*}
Therefore sum of variances may be written as given below:
\begin{multline*}
	\sum_{j=1}^{n}\Var\tilde{Z}_{n,j}=
    \sum_{j=1}^{n} \Var\left(\tilde{Z}_{n,j}\right)=\sum_{j=1}^n s_{n,j}^2+\Mean\sum_{j=1}^n r_{n,j}=
    \left[\sum_{j=1}^n r_{n,j}=r_n\right]=1+\Mean r_n.
\end{multline*}
By the above $r_n\to 0$ both in probability and in $L_1$, since
\begin{multline*}
	\Mean \abs{r_n}\leq \sum_{j=1}^n \Mean \abs{r_{n,j}} \leq 3\sum_{j=1}^{n_{\epsilon,\epsilon}-1} s_{n,j}^2 + 
    6\epsilon\sum_{j=n_{\epsilon,\epsilon}}^{n} s_{n,j}^2 
    \leq 3\sum_{j=1}^{n_{\epsilon,\epsilon}-1} s_{n,j}^2 + 
    6\epsilon \imply \\
    \limsup \Mean\abs{r_n}\leq 6\epsilon \stackrel{\epsilon \downarrow 0+0}{\Longrightarrow}\limsup \Mean\abs{r_n}=0.
\end{multline*}
Thus, $\sum_{j=1}^{n}\Var\tilde{Z}_{n,j} \to 1$ and 
\begin{align*}
	& \sum_{j=1}^n X_{n,j}^2 \stackrel{P}{\sim} \sum_{j=1}^{n}\tilde{Z}_{n,j}^2=1+r_n \tendsp 1;\\
    & \max_{1 \leq j \leq n} \abs{X_{n,j}}\leq \max_{1 \leq j \leq n} \abs{s_{n,j}}\frac{2}{\Var\left(\sum_{j=1}^{n}\tilde{Z}_{n,j}\right)} \sim 2\max_{1 \leq j \leq n} \abs{s_{n,j}}.
\end{align*}
Consequently, assumptions (a) and (b) of theorem \ref{t:McLeishCorr} hold provided \eqref{e:const1} holds.
\end{proof}

\begin{proof}[Proof of proposition \ref{prop:completeRandAlloc}.]

By remark (r3) of subsection \ref{ss:auxConcludingRems} for both rules it suffices to show that (c3) of subsection \ref{ss:auxConstraints} holds. We do this separately for each rule.

\smallskip\emph{Complete randomization.} $T_{n,j},j=1,\dots,n$, are i.i.d. Rademacher's r.v. Hence, $\forall \lambda >0 $,
\begin{equation*}
	\Mean \pi_n=\prod_{j=1}^n\Mean(1+\mathrm{i}\lambda Z_{n,j})=1.
\end{equation*}

\smallskip\emph{Random allocation rule.} First note that this rule produces $\binom{n}{\frac{n}{2}}$ equally likely permutations of $\frac{n}{2}$ of ones and $\frac{n}{2}$ of minus ones. Let $\mathrm{dom}((T_{n,1},\dots,T_{n,n}))= D_n \subset\{(k_1,\dots,k_n) \mid k_j\in \{0,1\}\}$ denotes that set. Then we can split it into two subsets $D_n^{+},D_n^{-}$ having equal numbers of elements and such that for each $(t_1,\dots,t_n)\in D_n^+$ there exists unique $(u_1,\dots,u_n)\in D_n^-$ having property $(u_1,\dots,u_n)=(-t_1,\dots,-t_n)$. Let $\lambda>0$ be fixed and $z_{t_1,\dots,t_n}=\prod_{j=1}^n\Mean(1+\mathrm{i}\lambda z_{j}t_j)$ for $z_1,\dots,z_n\in\Rd$ and $(t_1,\dots,t_n)\in D_n$. Denoting by $\bar{c}$ a conjugate and by $\Re c$ the real part of arbitrary $c\in\Cd$, the said then yields
\begin{multline*}
	\Mean\pi_n=\frac{1}{\binom{n}{\frac{n}{2}}}\left(\sum_{(t_1,\dots,t_n)\in D_n^+}z_{t_1,\dots,z_{t_n}}+
    \sum_{(t_1,\dots,t_n)\in D_n^-}z_{t_1,\dots,z_{t_n}}\right)=\\
    \frac{1}{\binom{n}{\frac{n}{2}}}\sum_{(t_1,\dots,t_n)\in D_n^+}\left(z_{t_1,\dots,z_{t_n}}+\bar{z}_{t_1,\dots,z_{t_n}}\right)=\frac{1}{\binom{n}{\frac{n}{2}}}\sum_{(t_1,\dots,t_n)\in D_n^+}2\Re z_{t_1,\dots,z_{t_n}}=\\2\Mean \mathbf{1}_{D_n^+}((T_{n,1},\dots,T_{n,n}))\Re \left(\prod_{j=1}^n(1+\mathrm{i}\lambda Z_{n,j})\right),
\end{multline*}
and by symmetry, $\Mean \pi_n=2\Mean \mathbf{1}_{D_n^-}((T_{n,1},\dots,T_{n,n}))\Re \left(\prod_{j=1}^n(1+\mathrm{i}\lambda Z_{n,j})\right)$. Adding the equalities one obtains an expression 
$\Mean \pi_n=\Mean \Re \left(\prod_{j=1}^n(1+\mathrm{i}\lambda Z_{n,j})\right)$.

Next, note that:
\begin{itemize}
\item for arbitrary $z_1,\dots,z_n\in\Rd$,
\begin{multline*}
	\Re \left(\prod_{j=1}^n(1+\mathrm{i}\lambda z_{j})\right)=1-\sum_{1\le j_1<j_2 \le n} z_{j_1}z_{j_2}+ 
    \sum_{1\le j_1<j_2<j_3<j_4 \le n} z_{j_1}z_{j_2}z_{j_3}z_{j_4}+\dots+\\
    (-1)^{(n-2)/2}
    \sum_{1\le j_1<\dots<j_{n-2} \le n} z_{j_1}\cdots z_{j_{n-2}}+(-1)^{n/2}\prod_{j=1}^n z_j;
\end{multline*}
\item for arbitrary $1<j_1<\dots<j_{2l}\leq n$ function $(T_{n,1},\dots,T_{n,n})\mapsto \prod_{k=1}^{2l}T_{{n,j_k}}$ attains values $\pm 1$ with equal probabilities.
\end{itemize}
Therefore $\Mean \pi_n=1$, constraint (c3) holds and implication (i) applies to this rule too.
\end{proof}

\begin{proof}[Proof of proposition \ref{prop:TBDSum}.]
Define
\begin{equation}\label{e:TBD_Xt}
	U_{n,j}=\mathbf{1}_{\{j \leq n-\tau\}}s_{n,j}(T_{n,j}-\Mean(T_{n,j}\mid \calF_{n,j-1})) ,\quad X_{n,j}=\frac{U_{n,j}}{\sqrt[]{\sum_{j=1}^n\Var \left(U_{n,j}\right)}},j=1,\dots,n.
\end{equation}
Since $\mathbf{1}_{\{j \leq n-\tau\}}=\mathbf{1}_{\{\max(S_{n,j-1}^{(-1)},S_{n,j-1}^{(1)}) < \frac{n}{2}\}}$, $U_{n,j},X_{n,j}$ are $\calF_{n,j-1}$ measurable. Moreover, $\forall j \Mean X_{n,j}=\Mean U_{n,j}=0$ and $\sum_{j=1}^{n}\Var X_{n,j}=1$. Consequently, $\{X_{n,j} \mid j=1,\dots,n\geq 1\}$ is a martingale difference array to which theorem \ref{t:McLeishCorr} may be applied. Next, note that 
\begin{equation*}
	\frac{U_{n,j}}{s_{n,j}} \ \Big |\ \tau = k \sim 
    \begin{cases}
		&\text{Rademacher's r.v. for } j\leq n-k;\\
        &\text{degenerate r.v. equal to 0 for } j > n-k.
	\end{cases}
\end{equation*}
Thus the law of total variance yields, 
\begin{equation*}\label{e:TBDvariances}
	\Var(U_{n,j})=s_{n,j}^2\Mean\Var(U_{n,j}/s_{n,j}\mid \tau)=s_{n,j}^2\Mean\mathbf{1}_{\{j \leq n-\tau\}}=s_{n,j}^2\Prob(\tau \leq n-j),\ j=1,\dots,\frac{n}{2}.
\end{equation*}
Consequently, conditions (a) and (b) of theorem \ref{t:McLeishCorr} read as
\begin{align*}
	&\text{(a) }\frac{\sum_{j=1}^{n-\tau} s_{n,j}^2}{\sum_{j=1}^{n/2} s_{n,j}^2+ \sum_{j=n/2+1}^{n} s_{n,j}^2\Prob(\tau \leq n-j)}\tendsp 1,\\
	&\text{(b) }\max_{1\leq j \leq n} \frac{\abs{s_{n,j}}}{\sqrt[]{\sum_{j=1}^{n/2} s_{n,j}^2+ \sum_{j=n/2+1}^{n} s_{n,j}^2\Prob(\tau \leq n-j)}} \tends{\ n\to \infty} 0,
\end{align*}
since $U_{n,1},\dots,U_{n,n-\tau}$ are i.i.d. Rademacher's variates. For $n$ sufficiently large $\sum_{j=1}^{n/2} s_{n,j}^2+ \sum_{j=n/2+1}^{n} s_{n,j}^2\Prob(\tau \leq n-j)$ becomes uniformly bounded away from zero because of (i). (c1) then implies (b) whereas rewriting (a) as
\begin{equation*}
	\frac{\sum_{j=1}^{n-\tau} s_{n,j}^2}{\sum_{j=1}^{n/2} s_{n,j}^2+ \sum_{j=n/2+1}^{n} s_{n,j}^2\Prob(\tau \leq n-j)}- 1 \tendsp 0,
\end{equation*}
one sees that it is equivalent to (ii). Hence, theorem \ref{t:McLeishCorr} applies and leads to the claim.
\end{proof}

\begin{proof}[Proof of proposition \ref{prop:TBDSum2}]
To give the proofs we need facts about the distribution of $\tau$ listed below.
\begin{itemize}
\item[(d1)] $\Mean \tau=\frac{n}{2^n}\binom{n}{n/2}\stackrel{n \to\infty}{\sim} \sqrt[]{n}$ (\cite{Lachin02}, subsection 3.4);
\item[(d2)] $\frac{\tau}{\sqrt[]{n}}\tendsd \abs{Z}, $ with $Z\sim N(0;1)$ (\cite{Rosenberger03}, lemma 1).
\end{itemize}

\noindent\emph{Proof of (s1).} By (d2), $\Prob(\tau \leq n-j)\approx 2\Phi(\sqrt{n}-\frac{j}{\sqrt{n}})-1>\Phi(1)-\frac{1}{2}>0$ uniformly for $j\in \{n/2+1,\dots,n-\sqrt{n}\}$ provided $n$ is large enough. Thus,
\begin{multline*}
	\sum_{j=1}^{n/2} s_{n,j}^2+ \sum_{j=n/2+1}^{n} s_{n,j}^2\Prob(\tau \leq n-j)\stackrel{n\to\infty}{>}\left(\Phi(1)-\frac{1}{2}\right)\sum_{j=1}^{n-\sqrt{n}} s_{n,j}^2\stackrel{n\to\infty}{\sim}\\
    \left(\Phi(1)-\frac{1}{2}\right)\sum_{j=1}^{n} s_{n,j}^2=\Phi(1)-\frac{1}{2},
\end{multline*}
since $\sum_{j=n-\sqrt{n}+1}^{n} s_{n,j}^2=\sqrt{n}o\left(\frac{1}{\sqrt{n}}\right)=o(1)$.

\noindent\emph{Proof of (s2).} Let $\xi_n=\sum_{j=n/2+1}^{n-\tau} s_{n,j}^2\Prob(\tau>n-j)-\sum_{j=n-\tau+1}^{n} s_{n,j}^2.$ Since $\abs{\xi_n}\leq 1$ is bounded, its convergence in probability to 0 is equivalent to convergence in $L_q$ for any fixed $q>0$, i. e., $\xi_n\tendsp 0 \Longleftrightarrow \Mean\abs{\xi_n}^q\tends{n\to\infty}0$. Take $q=1$. Then,
\begin{multline}\label{e:TBDbound1}
\Mean\Bigg | \sum_{j=n/2+1}^{n-\tau} s_{n,j}^2\bar{F}_\tau(n-j)-\sum_{j=n-\tau+1}^{n} s_{n,j}^2 \Bigg |\leq\\
	\Mean \sum_{j=n/2+1}^{n-\tau} s_{n,j}^2\bar{F}_\tau(n-j)+\Mean\sum_{j=n-\tau+1}^{n} s_{n,j}^2. 
\end{multline}
Next, setting $\max_{n/2\leq j\leq n}s_{n,j}^2=m_n$,
\begin{equation}\label{e:TBDbound2}
	\Mean\sum_{j=n-\tau+1}^{n} s_{n,j}^2 =m_n
    \Mean\sum_{j=n-\tau+1}^{n} \frac{s_{n,j}^2}{m_n} \leq m_n \Mean \tau \stackrel{n\to\infty}{\sim} m_n \sqrt[]{n}
\end{equation}
and 
\begin{multline*}
	\Mean \sum_{j=n/2+1}^{n-\tau} s_{n,j}^2\bar{F}_\tau(n-j)\leq 
    m_n\Mean \left(\sum_{j=n/2}^{n-\tau} \bar{F}_\tau(n-j)\right)=\\m_n\Mean \left(\sum_{j=\tau}^{n/2} \bar{F}_\tau(j)\right)=
    m_n\sum_{k=1}^{n/2} \Prob(\tau=k)\left(\sum_{j=k}^{n/2} \bar{F}_\tau(j)\right)=\\
    m_n\sum_{j=1}^{n/2} \bar{F}_\tau(j)\left(\sum_{k=1}^{j} \Prob(\tau=k)\right)=
    m_n\sum_{j=1}^{n/2} \bar{F}_\tau(j)F_{\tau}(j).
\end{multline*}
Since the limiting distribution of $\frac{\tau}{\sqrt[]{n}}$ is continuous, convergence of the c.d.f. is uniform on the whole real line. Therefore denoting by $\Phi(x)$ the c.d.f. of $Z\sim N(0;1)$ and making use of the well known asymptotic relationship $1-\Phi(x)\sim\frac{1}{x}\e^{-\frac{x^2}{2}}$,
\begin{multline}\label{e:TBDbound3}
	m_n\sum_{j=1}^{n/2} \bar{F}_\tau(j)F_{\tau}(j) \stackrel{n\to\infty}{\sim}
    4m_n\sum_{j=1}^{n/2} \left(\Phi\left(\frac{j}{\sqrt[]{n}}\right)-\frac{1}{2}\right)\left(1-\Phi\left(\frac{j}{\sqrt[]{n}}\right)\right)\leq\\
    m_n\left(\sqrt{2n\ln n} + \sum_{j=\sqrt{2n\ln n}+1}^{n/2}\left(1-\Phi\left(\frac{j}{\sqrt[]{n}}\right)\right)\right)\stackrel{n\to\infty}{\sim} m_n\sqrt{2n\ln n}.
\end{multline}
Combination of \eqref{e:TBDbound1}--\eqref{e:TBDbound3} thus yields
\begin{equation*}
	\Mean\Bigg | \sum_{j=n/2+1}^{n-\tau} s_{n,j}^2\bar{F}_\tau(n-j)-\sum_{j=n-\tau+1}^{n} s_{n,j}^2 \Bigg |\leq const.\cdot m_n \sqrt[]{n\ln n}=o(1).
\end{equation*}
\end{proof}

\bibliographystyle{gNST}
\bibliography{references}

\end{document}